\documentclass[10pt]{article}
\usepackage[english]{babel}
\usepackage{amsmath,amsthm}
\usepackage{amsfonts}
\usepackage[none]{hyphenat}
\usepackage{xcolor}
\usepackage[font={small}]{caption}
\usepackage[explicit]{titlesec}
\usepackage{anyfontsize}
\usepackage{enumerate}
\titleformat{\section}{\Large}{\textbf{\thesection .}}{1em}{\textbf{#1}}

\addto{\captionsenglish}{}
\newtheorem{thm}{Theorem}[section]
\newtheorem{cor}[thm]{Corollary}
\newtheorem{lem}[thm]{Lemma}
\newtheorem{prop}[thm]{Proposition}
\theoremstyle{definition}
\newtheorem{defn}[thm]{Definition}
\theoremstyle{remark}

\numberwithin{equation}{section}

\renewenvironment{proof}{{\vspace{-1em}\quad \textit{Proof.}}}{\hfill$\square$}
\begin{document}
\renewcommand{\thefootnote}{\fnsymbol{footnote}}
\title{Forbidden Patterns and the Alternating Derangement Sequence}%
\author{\vspace{-3,5\baselineskip}Enrique Navarrete$^{\ast}$}\footnotetext[1]{Grupo ANFI, Universidad de Antioquia, Medell\'in, Colombia.}%
\date{}%
\makeatletter
\def\maketitle{%
\bgroup
\par\centering{\LARGE\@title}\\[3em]%
\ \par
{\@author}\\[4em]%
\egroup
}

\maketitle
\begin{abstract}
  In this note we count linear arrangements that avoid certain\linebreak patterns and show their connection to the derangement numbers. We introduce the sequence $\langle D_n\rangle$, which counts the linear arrangements that avoid the patterns $12, 23,\ldots, (n-1)n, n1$, and show that this sequence almost follows the derangement sequence itself since the number of its odd terms is one more than the derangement numbers while the number of its even terms is one less. We also express the derangement numbers in terms of these and other linear arrangements.  Finally, we relate these arrangements to permutations and show that both the arrangements\linebreak described above as well as deranged permutations are equidistributed, in the sense defined in Section 5.\\[1em]
\end{abstract}
----------------------------------------------------------------
\section{Introduction}

In this note we will be counting the number of ways in which the integers $1, 2, \ldots, n$ can be arranged in a line so that some patterns are forbidden or blocked. We will be mainly concerned with the arrangements that block the patterns $12, 23, \ldots , (n-1)n, n1$. This problem can be naturally phrased in terms of permutations of the integer set $S = \{1, 2, \ldots, n\}$ that avoid the pairs just mentioned. Hence in Section 5 we will be identifying linear arrangements with some permutations in the symmetric group.

Since we will be discussing several forbidden patterns, the following definitions are in order:\vspace{-1em}
\begin{description}
  \item[$S_n$] :=  the set of all linear arrangements on the set of integers $\{1,2, \ldots, n\}$.
  \item[$\{d_n\}$]  :=  the set of  linear arrangements of  $S_n$ that avoid the patterns\linebreak $12, 23, \ldots, (n-1)n$.
  \item[$\{d_{n1}\}$] :=  the subset of  linear arrangements of $\{d_n\}$ that allow the pattern $n1$.
  \item[$\{D_n\}$] :=  the set of linear arrangements of  $S_n$ that avoid the patterns\linebreak $12, 23, \ldots, (n-1)n, n1$.
  \item[$\{Der_n\}$] :=  the set of deranged arrangements of the integers $\{1,2, \ldots, n\}$.
  \item[$d_n$] := cardinality of the set $\{d_n\}$.
  \item[$d_{n1}$] := cardinality of the set $\{d_{n1}\}$.
  \item[$D_n$] := cardinality of the set $\{D_n\}$.
  \item[$\langle D_n\rangle$] := the sequence of numbers $D_n$, $n = 0,1,2,\ldots$
  \item[$Der_n$] := the nth derangement number, \textit{ie}.
  \begin{equation}\label{eq1}
    Der_{n}=n!\sum_{k=0}^{n}\frac{(-1)^{n}}{k!}.
  \end{equation}
\end{description}

Note that in all cases we will be discussing linear arrangements (as opposed to circular ones).

\section{Main Lemmas and Proposition}

We first see that there are $n!$ possible linear arrangements in $S_n$, but not all of them will be valid for our counting purposes. The largest subset of arrangements we will be considering is $\{d_n\}$, the set of  linear arrangements of $S_n$ that avoid the patterns $12, 23, \ldots, (n-1)n$.   Hence the set of forbidden patterns is $S_n - \{d_n\}$. This set will be used at the end of the following main result.

\begin{lem}\label{lemm21}
For $n \geq 1$, $d_n = D_{n-1} + D_n$.
\end{lem}

\begin{proof}
The Lemma is true for the case $n = 1$ if we define $D_0 = 0$, $D_1 = 1$, $d_1 = 1$.  The cases $n = 2, 3$ are easily checked, since $d_2 = 1$, $D_2 = 0$, $d_3 = 3$, $D_3 = 3$, (see Table \ref{tabA1} in the Appendix for arrangements in $\{D_n\}$ and $\{d_n\}$ up to $n = 5$).

We will show in general that $\{d_n\} = \{D_{n-1}\} \cup \{D_n\}$ is a disjoint union by actually showing that $\{d_n\} = \{d_{n1}\} \cup \{D_n\}$ is a disjoint union and that\linebreak $\{d_{n1}\}\leftrightarrow\{D_{n-1}\}$ is a bijection.

To show the bijection between the sets $\{d_{n1}\}$ and $\{D_{n-1}\}$ we will construct arrangements in $\{d_{n1}\}$ starting from arrangements in $\{D_{n-1}\}$ and the map $\varphi:\{D_{n-1}\}\rightarrow\{d_{n1}\}$ that inserts the digit $n$ in front of the digit 1 in linear arrangements in $\{D_{n-1}\}$. Note that arrangements in $\{D_{n-1}\}$ only contain digits $1,2, \ldots, n-1$, but this map produces arrangements of length $n$ in $\{d_{n1}\}$. We claim this map creates all the arrangements in $\{d_{n1}\}$ \textit{ie}. arrangements in $\{d_n\}$ that allow the pattern $n1$. We have to show that the arrangements obtained this way are distinct, that they contain no forbidden patterns $12, 23, \ldots, (n-1)n$, that $\{d_n\}$ is the disjoint union $\{d_{n1}\} \cup \{D_n\}$, and that all the arrangements in $\{d_{n1}\}$ are formed by the map described above.  We show each of these properties below.

\begin{enumerate}
  \item The sets $\{d_{n1}\}$ and $\{D_n\}$ contain arrangements which are $n$ digits long and are disjoint since, by construction, arrangements in $\{d_{n1}\}$ allow the pattern $n1$ while those in $\{D_n\}$ avoid it.
  \item To show that the linear arrangements in $\{d_{n1}\}$ produced by the map are valid, $ie$.\hspace{-0.8mm} contain no forbidden patterns $12, 23, \ldots, (n-1)n$, we see that the only problematic arrangements would be those for which, when inserting an $n$ before a 1, we create the pattern $(n-1)n1$, which is forbidden due to the pattern $(n-1)n$. However, this will never happen since all the arrangements in $\{D_{n-1}\}$, from which $\{d_{n1}\}$ is constructed, avoid the pattern $(n-1)1$.
  \item To show the map is 1-1 is trivial, since by removing the digit $n$ in\linebreak arrangements in $\{d_{n1}\}$ we get back arrangements in $\{D_{n-1}\}$, which are distinct and valid, by the previous step.
  \item To show that the map $\{D_{n-1}\}\rightarrow\{d_{n1}\}$ is onto, we need to show that all the arrangements in $\{d_{n1}\}$ are produced from $\{D_{n-1}\}$ by the map described above and that there is no other way to produce them.
\end{enumerate}

To show this, suppose we want to create arrangements in $\{d_{n1}\}$ from the larger set $\{d_{n-1}\}$ (as opposed to $\{D_{n-1}\}$). If we start from arrangements in $\{d_{n-1}\}$, there are members of this set that allow the pattern $(n-1)1$, hence inserting the digit $n$ before 1 will produce the pattern $(n-1)n$, which is forbidden, so this case is settled.

Finally, since $S_{n-1} = \{d_{n-1}\} \cup \{S_{n-1} - d_{n-1}\}$ is a disjoint union, the last possibility would be that if we insert the digit $n$ before 1 in the forbidden arrangements in $\{S_{n-1} - d_{n-1}\}$ we might break a forbidden pattern and create a valid arrangement in $\{d_{n1}\}$ with the pattern $n1$ in it. However, this cannot happen since when we insert $n$ before 1, the forbidden patterns\linebreak $12, 23, \ldots, (n-1)n$ will remain intact in the set $\{S_{n-1} - d_{n-1}\}$ of forbidden arrangements. Therefore the only way to produce valid arrangements in $\{d_n\}$ that allow the pattern $n1$ is to start with the set $\{D_{n-1}\}$ and form the\linebreak arrangements in $\{d_{n1}\}$ by the map described above.

Since the previous steps showed that $\{d_n\} = \{d_{n1}\} \cup \{D_n\}$ is a disjoint union and that $\{d_{n1}\}\leftrightarrow\{D_{n-1}\}$ is a bijection, this implies the Lemma for the set cardinalities.
\end{proof}

We record explicitly the result of the bijection below.

\begin{cor}\label{coro22}
There are exactly $D_{n-1}$  linear arrangements in $\{d_n\}$ that contain the pattern $n1$.
\end{cor}

\begin{proof}
This follows from the previous Lemma by the bijection\linebreak $\{d_{n1}\}\leftrightarrow\{D_{n-1}\}$.
\end{proof}

Note that the counting relation proved above is very similar to a well known formula for $d_n$, which we state below as a lemma:

\begin{lem}\label{lemm23}
For all $n \geq 1$, $d_n = Der_n + Der_{n-1}$.
\end{lem}

\begin{proof}
Since $d_n$ counts the linear arrangements that avoid the patterns $12, 23, \ldots, (n-1)n$, it is easy to show that the there are a total of
\[\binom{n-1}{k}(n-k)!\]
forbidden patterns of length $k$ since the combinatorial term counts the number of ways to get such patterns while the term $(n-k)!$ counts the permutations of the patterns and the remaining elements (note that patterns are not necessarily disjoint but may overlap in pattern strings). Then using inclusion-exclusion we have that
\[d_n=\sum_{k=0}^{n-1}(-1)^k\binom{n-1}{k}(n-k)!\text{,}\]
which is equal to the right-hand side of the equation in the Lemma by direct computation.
\end{proof}

Note also that the expression for $d_n$  above is similar to the derangement formula,
\[Der_n=\sum_{k=0}^{n}(-1)^k\binom{n}{k}(n-k)!.\]
Now we get to the main proposition:
\begin{prop}\label{propo24}
For all $n \geq 1$, $D_n = Der_n + (-1)^{n-1}$.
\end{prop}

\begin{proof}
Combining the two previous lemmas we have that\linebreak $D_n = Der_n + Der_{n-1} - D_{n-1}$. Iterating this expression repeatedly yields the Proposition.  Note that we have $n-1$, not $n$ in the exponent of -1 since the iteration stops after $n-1$ steps at $Der_2 = 1$.
\end{proof}

Note from the Proposition that by adding $(-1)^n$ to both sides we get\linebreak $Der_n = D_n + (-1)^n$, so that $Der_n - D_n = (-1)^n$. This means that the numbers $D_n$ can be obtained by adding in the usual derangement formula, Equation \ref{eq1} above, up to the term $n-1$, leaving out the last term, $(-1)^n$. That is, for the number of linear arrangements that avoid the patterns $12, 23, \ldots, (n-1)n, n1$, we have that
\begin{equation}\label{eq2}
  D_n=n!\sum_{k=0}^{n-1}\frac{(-1)^k}{k!}.
\end{equation}
This creates the alternating behavior referred to above, in the sense that the number of these arrangements is one less than the derangement number when $n$ is even, and one more than the derangement number when $n$ is odd.  Hence we have the sequence $\langle D_n\rangle$ starting at $\mbox{n = 0}$ that runs like\linebreak $\{0, 1, 0, 3, 8, 45, 264, 1855, 14832,\ldots\}$ (see Table \ref{tabA2} in the Appendix).

Notice that when we combine Lemmas \ref{lemm21} and \ref{lemm23} we get the equality
\[d_n=D_n+D_{n-1}=Der_n +Der_{n-1},\]
which produces the following sequence of equations starting at $n = 1$:
\begin{align*}
  &d_1 =  1 + 0 = 0 + 1 = 1 & &(n = 1);\\
  &d_2 =  0 + 1 = 1 + 0 = 1 & &(n = 2);\\
  &d_3 =  3 + 0 = 2 + 1 = 3 & &(n = 3);\\
  &d_4 =  8 + 3 = 9 + 2 = 11 & &(n = 4);\\
  &d_5 =  45 + 8 = 44 + 9 = 53 & &(n = 5)\ldots
\end{align*}

\section{Recursions followed by linear arrangements in $\boldsymbol{\{D_n\}}$}

Notice from Table \ref{tabA2} that the numbers $D_n$ are all divisible by $n$. This is not a coincidence but a result of the following lemmas and recursions for linear arrangements in $\{D_n\}$.

\begin{lem}\label{lemm31}
For all $n \geq 1$, $D_n = n Der_{n-1}$.
\end{lem}

\begin{proof}
From Proposition \ref{propo24} we have that  $D_n  =  Der_n + (-1)^{n-1}$,\linebreak $n \geq 1$. If we add $(-1)^n$  to both sides we have that  $Der_n = D_n + (-1)^n$. Then using a well-known recursion from the derangement numbers, namely $Der_n = n Der_{n-1} + (-1)^n$, $n \geq 1$, we equate both sides to get the equation in the Lemma.
\end{proof}

Now we get a recursion similar to the one used in the proof of the Lemma for arrangements in $\{D_n\}$.

\begin{cor}\label{coro32}
The recursion $D_n  = n[D_{n-1} + (-1)^{n-1}]$, $n \geq 1$, holds for linear arrangements in $\{D_n\}$.
\end{cor}

\begin{proof}
From the proof of the Lemma above we have that $Der_n = D_n + (-1)^n$, hence $Der_{n-1} = D_{n-1} + (-1)^{n-1}$. The result is immediate upon substitution in the right-hand side of Lemma \ref{lemm31}, $D_n = n Der_{n-1}$.
\end{proof}

Combining most of the lemmas we have so far, we can also get the following recursion for $D_n$, which is similar to the well-known recursion\linebreak $Der_n  = (n-1)[Der_{n-1} + Der_{n-2}]$, $n \geq  2$, for the derangement numbers.

\begin{cor}\label{coro33}
The recursion $D_n  = (n-1)[D_{n-1} + D_{n-2}] + (-1)^{n-1}$, $n \geq 2$, holds for linear arrangements in $\{D_n\}$.
\end{cor}

\begin{proof}
\hspace{-9mm}Since \mbox{$D_n=Der_n+(-1)^{n-1}$} by Proposition \ref{propo24}, and\linebreak $Der_n = (n-1)[Der_{n-1} + Der_{n-2}]$, $n \geq 2$, for the derangement numbers, we have that $D_n  = (n-1)[Der_{n-1} + Der_{n-2}] + (-1)^{n-1}$.
By Lemma \ref{lemm23} we have that $d_{n-1} = Der_{n-1} + Der_{n-2}$, and by Lemma \ref{lemm21}, that $d_{n-1} = D_{n-1} + D_{n-2}$.

\end{proof}

We summarize below the recursions followed by arrangements in $\{D_n\}$ along with the corresponding recursions for derangements $\{Der_n\}$.

\begin{table}[h!]
  \centering
    {
    \hspace*{-1cm}
    \begin{tabular}{|c|c|c|}
     \hline
     {\small Linear arrangements $\{D_n\}$} & {\small Derangements $\{Der_n\}$} & {\small Reference}\\
     \hline
     $D_n=n\left[D_{n-1}+(-1)^{n-1}\right]$ & $Der_n=nDer_{n-1}+(-1)^{n}$ & {\small Coroll. \ref{coro32}}\\
     $D_n=(n-1)\left[D_{n-1}+D_{n-2}\right]+(-1)^{n-1}$ & $Der_n=(n-1)\left[Der_{n-1}+Der_{n-2}\right]$ & {\small Coroll. \ref{coro33}}\\
     \hline
   \end{tabular}}
  \caption{ Recursions valid for linear arrangements $\{D_n\}$ and derangements $\{Der_n\}$.}\label{tab1}
\end{table}

The recursions and lemmas for $\{D_n\}$ have some interesting consequences that we state  below as the following corollaries.

\begin{cor}\label{coro34}
The number $D_n$ is divisible by $n$.
\end{cor}

\begin{proof}
This is just Lemma \ref{lemm31}, $D_n = n Der_{n-1}$.
\end{proof}

We will see that even more can be said from this equation in the next section.

Corollary \ref{coro34} in turn implies the following:

\begin{prop}
$Der_n - 1$ is divisible by $n$ for $n$ even $(n > 0)$, whereas $Der_n + 1$ is divisible by $n$ for $n$ odd.
\end{prop}

\begin{proof}
Immediate from Proposition \ref{propo24} and Corollary \ref{coro34}. By adding $(-1)^n$ to both sides of Proposition \ref{propo24} we have that $Der_n = D_n + (-1)^n$, and the fact that $n$ divides $D_n$ (Corollary \ref{coro34}) gives the result.
\end{proof}

We will get more divisibility properties from the following lemma:

\begin{lem}\label{lemm36}
The derangement numbers $Der_n$ can be expressed in terms of the number of linear arrangements of the set of integers $\{1,2, \ldots, n\}$ that avoid the patterns $12, 23, \ldots, (n-1)n$, as $Der_n = (n-1)d_{n-1}$, $n \geq 2$.
\end{lem}

\begin{proof}
By substituting the counting relation $d_{n-1} = Der_{n-1} + Der_{n-2}$ \mbox{(Lemma \ref{lemm23})} into the recursion $Der_n  = (n-1)[Der_{n-1} + Der_{n-2}]$.
\end{proof}

From this we get the following.

\begin{cor}\label{coro37}
The number $Der_n$ is divisible by $n-1$, $n \geq 2$.
\end{cor}

\begin{proof}
This is just a rephrasing of Lemma \ref{lemm36}.
\end{proof}

We will see that more can be said from these divisibility properties in the next section.

\section{Equidistribution properties for linear arrangements and derangements}

Now that we have found recursions for linear arrangements in $\{D_n\}$, we move on to discuss distribution properties of these arrangements. We first state two brief definitions.

\begin{defn}
A \textit{class} of linear arrangements is a subset of arrangements of $\{D_n\}$, $\{d_n\}$ or $\{d_{n1}\}$ that start with the same digit.
\end{defn}

\begin{defn}
We say that a set of linear arrangements is \textit{equidistributed} if the set can be partitioned into nonempty classes having the same cardinality.
\end{defn}

\vspace{-1em}
Now we have the following proposition and corollaries:

\begin{prop}\label{propo43}
The linear arrangements that avoid the patterns\linebreak $12, 23, \ldots, (n-1)n, n1$, \textit{ie}. $\{D_n\}$ are equidistributed for all $n \geq 3$.
\end{prop}

\begin{proof}
From Lemma \ref{lemm31} we have that $D_n = n Der_{n-1}$.  In fact, this equation may be read as $\{D_n\}$ consisting of $n$ nonempty classes starting with digits $1, 2, \ldots, n$, each class consisting of $Der_{n-1}$ members.  The classes in $\{D_n\}$ are not empty since we can start each of the $n$ classes with the patterns\linebreak $1n, 21, 32, \ldots, n(n-1)$, which are valid in $\{D_n\}$.
\end{proof}

The fact that $\{D_n\}$ can be partitioned into $n$ classes of size $Der_{n-1}$ is important enough to record it as a corollary:

\begin{cor}
The number of members in each class of $\{D_n\}$ is exactly $Der_{n-1}$.
\end{cor}

\vspace{-1em}
Hence we see that there is yet another object being counted by the derangement numbers, namely the cardinality or number of members in each class of $\{D_n\}$.

As an example of the Proposition and Corollary, we have $D_4 = 8 = 4\cdot 2$, hence $\{D_4\}$ consists of 4 classes of linear arrangements, each class consisting of 2 members.  Similarly, $D_5 = 45 = 5\cdot 9$, hence $\{D_5\}$ consists of 5 classes of arrangements with 9 members in each class.

Now we turn our attention to deranged arrangements, or derangements $\{Der_n\}$.

We have the following proposition:

\begin{prop}\label{propo45}
The deranged linear arrangements $\{Der_n\}$ are equidistributed for all $n \geq 2$.
\end{prop}

\begin{proof}
We see that for derangements, the first class (which starts  with 1) is empty, so we only have $n-1$ nonempty classes. Then from Corollary \ref{coro37} we have that $n-1$ divides $Der_n$ and the result follows.
\end{proof}

The fact that $\{Der_n\}$ can be partitioned into $n-1$ classes of size $d_{n-1}$ is important enough to record it as a corollary:

\begin{cor}
The number of members in each nonempty class of $\{Der_n\}$ is exactly $d_{n-1}$.
\end{cor}

\begin{proof}
Lemma \ref{lemm36} and Corollary \ref{coro37} give these counting relations.
\end{proof}

As an example of the Corollary, we have $Der_4 = 9 = 3\cdot 3$, hence $\{Der_4\}$ consists of 3 classes of arrangements (starting with 2, 3 and 4), each class consisting of 3 members. Similarly, $Der_5 = 44 = 4\cdot 11$, hence $\{Der_5\}$ consists of 4 classes of arrangements with 11 members in each class.

We also have the following:

\begin{lem}\label{lemm47}
For $n \geq 2$, the class in $\{d_{n1}\}$ that starts with 1 (\textit{ie}. the first class) is empty.
\end{lem}

\begin{proof}
When we form the arrangements in $\{d_{n1}\}$ from $\{D_{n-1}\}$ by the\linebreak $n1$-mapping from Lemma \ref{lemm21}, we know that the class in $\{D_{n-1}\}$ that starts with 1 will be mapped into the class that starts with $n1$ in $\{d_{n1}\}$, while the other classes $i$, $i = 2, 3, \ldots, n-1$ get mapped to the same class. This means that no class in $\{D_{n-1}\}$ gets mapped to the first class in $\{d_{n1}\}$, so this class will be empty.
\end{proof}

\begin{lem}\label{lemm48}
For $n \geq 2$, the number of members in each nonempty class of $\{d_{n1}\}$ is   exactly $Der_{n-2}$.
\end{lem}

\begin{proof}
By Lemma \ref{lemm31}, we have that $D_n = n Der_{n-1}$. By Corollary \ref{coro22}, since $d_{n1} = D_{n-1}$, we have that $d_{n1} = D_{n-1}  = (n-1) Der_{n-2}$.
\end{proof}

\begin{prop}
The linear arrangements in $\{d_n\}$ that contain the pattern $n1$ \textit{ie}. $\{d_{n1}\}$ are equidistributed for $n \geq 2$.
\end{prop}

\begin{proof}
Lemma \ref{lemm47} shows that there are exactly $n-1$ nonempty classes in $\{d_{n1}\}$. By Lemma \ref{lemm48}, we have that $d_{n1} = (n-1) Der_{n-2}$. Hence $\{d_{n1}\}$ consists of $n-1$ nonempty classes, each class consisting of $Der_{n-2}$ arrangements.
\end{proof}

We record explicitly this divisibility property:

\begin{cor}
For $n \geq 2$, $n-1$ divides $d_{n1}$.
\end{cor}
%

\vspace{-1em}An example of this relationship is given by the arrangements in $\{D_n\}$ for $n = 5$ that allow the pattern 51. The class that starts with 1 is empty, and the remaining  $n-1 = 4$ classes have $Der_{n-2} = Der_3 = 2$ members each.

Note also that for $n = 2$, we have that $\{d_2\} = \{d_{21}\}$ since the only arrangement in this case is $\{21\}$, whereas for $n = 3$, we have that $\{D_3\} = \{d_3\}$, hence $\{d_{31}\}$ is empty and there are no arrangements that contain the pattern 31 (see \mbox{Table \ref{tabA2}} in the Appendix).

So far we have seen that the linear arrangements $\{D_n\}$, the deranged\linebreak arrangements $\{Der_n\}$, and the arrangements in $\{d_n\}$ that contain the pattern $n1$ are equidistributed, with the last two of them containing only $n-1$ nonempty classes.

We summarize the different properties of these arrangements in the following table:

\begin{table}[h!]
  \centering
  {
  \begin{tabular}{|c|c|c|c|c|}
    \hline
    \parbox{1.9cm}{\centering {\small Arrangement Type}}& \parbox[c][1.3cm]{1.6cm}{\centering {\small Number of Nonempty Classes}} & \parbox{1.2cm}{\centering {\small Size of Class}} & \parbox{3.5cm}{\centering {\small Relevant Equation}} & \parbox{1.7cm}{\centering {\small Numbered Reference}}\\
    \hline
    $D_n$ & $n$ & $Der_{n-1}$ &$D_n=nDer_{n-1}$ & {\small Lemma \ref{lemm31}}\\
    $Der_n$ & $n-1$ & $d_{n-1}$ &$Der_n=(n-1)d_{n-1}$ & {\small Lemma \ref{lemm36}}\\
    $d_{n1}$ & $n-1$ & $Der_{n-2}$ & $d_{n-1}=(n-1)Der_{n-2}$& {\small Lemma \ref{lemm48}}\\
    \hline
  \end{tabular}}
  \caption{Properties satisfied by $\{D_n\}$, $\{d_{n1}\}$, and the derangements $\{Der_n\}$.}\label{tab2}
\end{table}

We will now see, however, that the linear arrangements in $\{d_n\}$ are not\linebreak equidistributed.

\begin{prop}
The linear arrangements that avoid the patterns\linebreak $12, 23, \ldots, (n-1)n$, \textit{ie}. $\{d_n\}$ are not equidistributed for $n > 3$.
\end{prop}

\begin{proof}
We know from Lemma \ref{lemm21} that $\{d_n\} = \{D_n\}\cup \{d_{n1}\}$ and from Proposition \ref{propo43} that $\{D_n\}$ is equidistributed, with no class empty. When we form arrangements in $\{d_{n1}\}$ from $\{D_{n-1}\}$ by the $n1$-mapping, Lemma \ref{lemm47} shows that no class in $\{D_{n-1}\}$ gets mapped to the first class in $\{d_{n1}\}$, while the other classes $i$, $i = 2, 3,\ldots, n-1$ in $\{D_{n-1}\}$ get mapped to the same class in $\{d_{n1}\}$, hence these classes get increased by at least one arrangement. This proves the Proposition.
\end{proof}

\section{Results for Permutations}

Now we turn our attention to permutations themselves. Note that since there is a natural correspondence between linear arrangements and permutations in one-line notation, we want to examine how Propositions \ref{propo43} and \ref{propo45} behave in terms of permutations. To do so we need a preliminary result.

\begin{prop}\label{propo51}
For $n \geq 2$, there is a bijection between linear arrangements in $\{d_n\}$ and permutations in one-line notation.
\end{prop}

\begin{proof}
Let us consider any linear arrangement $\lambda$ in $\{d_n\}$ as representing a permutation in one-line notation and the mapping $\sigma$ of the arrangement $12\ldots n$ by the permutation. Then the digits in $\lambda$ represent the second line in the permutation, and since such digits avoid the patterns $j$, $j+1$, $j = 1, \ldots, n-1$, the digits $i$, $i+1$ in the line above them \textit{ie}. the digits in $12\ldots n$ get mapped to an arrangement $\lambda'$ that avoids successive patterns, hence $\lambda'$ belongs either to $\{d_n\}$ or $\{D_n\}$.  However, the map $\sigma(1) = k +1$, $\sigma(n) = k$ is not ruled out so the pattern $n1$ may be formed, which is not allowed in $\{D_n\}$. The fact that the map is bijective follows by inverse mappings of permutations.
\end{proof}

As an example of the Proposition, consider the linear arrangement 42153 in $\{d_5\}$. If we consider this arrangement as a permutation in one-line notation, then the arrangement maps 12345 to 32514, which contains the pattern 51, not allowed in $\{D_5\}$.

Now we define permutations that avoid patterns $12, 23, \ldots, (n-1)n$, $n1$ in terms of the linear arrangements they produce.

\begin{defn}\label{defi52}
A permutation (in one-line notation) \textit{avoids the patterns}\linebreak $12, 23, \ldots, (n-1)n$, $n1$ if it maps the arrangement $12\ldots n$  to  an arrangement in $\{D_n\}$.
\end{defn}

\begin{prop}
The permutations (in one-line notation) that avoid the\linebreak patterns $12, 23, \ldots, (n-1)n$, $n1$ are not equidistributed for $n \geq 3$.
\end{prop}

\begin{proof}
From the proof of Proposition \ref{propo51}, we see that arrangements in $\{D_n\}$ may map to arrangements in $\{d_{n1}\}$ under the one-line permutation map, so $\{D_n\}$, which is equidistributed, does not map to itself. In particular,\linebreak arrangements in $\{D_n\}$ that start with 1 will never map $12\ldots n$ to arrangements in $\{d_{n1}\}$ under the one-line permutation map, unlike arrangements from other classes.
\end{proof}

Next we define deranged permutations in similar way as in Definition \ref{defi52}.

\begin{defn}
A permutation (in one-line notation) is \textit{deranged} if it maps the arrangement $12\ldots n$  to an arrangement in $\{Der_n\}$.
\end{defn}

\vspace{-1em}Using this definition, we see that Proposition \ref{propo45} does hold in terms of both deranged linear arrangements and deranged permutations.

\begin{prop}\label{propo55}
For $n \geq 2$, there is a bijection between deranged linear arrangements $\{Der_n\}$ \textup{(}ie. derangements\textup{)} and permutations in one-line notation.
\end{prop}

\begin{proof}
If we consider deranged arrangements in $\{Der_n\}$ as representing permutations in one-line notation and the mapping of $12\ldots n$ by the\linebreak permutations, then, as in Proposition \ref{propo51}, the digits in the arrangements\linebreak represent the second line in the permutations. Since these digits have no fixed points, the digits $i$, $i+1$ in $12\ldots n$ get mapped to derangements. The fact that the map is bijective follows by the existence of inverse mappings for permutations.
\end{proof}

As an example of the Proposition, consider the deranged arrangement 3142\linebreak in $\{Der_4\}$. If we consider this arrangement as a permutation in one-line notation, then the arrangement 1234 gets mapped to 2413, which is in $\{Der_4\}$.

\begin{prop}
The deranged permutations are equidistributed for all $n \geq 2$.
\end{prop}

\begin{proof}
From the bijection in Proposition \ref{propo55}, and the fact that deranged arrangements are equidistributed (Proposition \ref{propo45}).
\end{proof}

We end this note with an easy application of one of the counting relationships derived previously.
\section{Another object counted by $\boldsymbol{\langle D_n\rangle }$}

We see that the sequence $\langle D_n\rangle = \{0, 1, 0, 3, 8, 45, 264, 1855, 14832,\ldots\}$ not only counts linear arrangements that avoid the patterns $12, 23, \ldots, (n-1)n, n1$, but also counts the permutations that have exactly one fixed point. This can be seen directly from inclusion-exclusion, since the formula that holds for exactly $m$ conditions is given by:
\[E_m=S_m-\binom{m+1}{1}S_{m+1}+\binom{m+2}{2}S_{m+2}+\cdots+(-1)^{n-m}\binom{n}{n-m}S_n.\]
For $m = 1$, this gives:
\[E_1=S_1-\binom{2}{1}S_2+\binom{3}{2}S_3+\cdots+(-1)^{n-1}\binom{n}{n-1}S_n.\]
Upon substitution of
\[S_j=\binom{n}{j}(n-j)!, \quad j= 1,\ldots,n,\]
this reduces to
\[E_1=n!\sum_{k=0}^{n-1}\frac{(-1)^k}{k!},\]
which is just Equation \ref{eq2} for $D_n$ right after Proposition \ref{propo24}.

Alternatively, this can be seen from the following final proposition.

\begin{prop}
The permutations that have exactly one fixed point are\linebreak counted by the sequence $\langle D_n\rangle$.
\end{prop}

\begin{proof}
If we let $h(n,k)$ be the number of permutations of $n$ objects with exactly $k$  fixed points, we see that  $h(n,1) = n h(n-1,0)$  holds since the left hand side counts the number of permutations with exactly one fixed point, while the right-hand side counts the $n$ ways to choose a fixed point and the number of ways to derange $n-1$ other elements. But the right-hand side is just $n Der_{n-1}$, which we have seen is equal to $D_n$  by Lemma \ref{lemm31}.  Hence $h(n,1) = D_n$ and the proof follows.
\end{proof}

\section{Conclusions}

We have counted the number of linear arrangements that avoid certain patterns for the sets $\{D_n\}$, $\{d_n\}$, and $\{d_{n1}\}$, and have also obtained equations for the derangement numbers $Der_n$ in terms of them. We see that both $\{D_n\}$ and $\{d_{n1}\}$ share the equidistribution property of the deranged arrangements, $\{Der_n\}$, and that if we define deranged permutations using one-line notation, these permutations are equidistributed as well.

We also see that the sequence $\langle D_n\rangle$  not only counts linear arrangements that avoid the patterns $12, 23, \ldots, (n-1)n, n1$, but also the permutations that have exactly one fixed point. Even though this result can be obtained by\linebreak inclusion-exclusion, it is an easy consequence of the counting relations developed. An explicit bijection between $\{D_n\}$ and the permutations that have exactly one fixed point would be interesting.

Finally, we see that the sequence $\langle D_n\rangle$ is very conveniently described since it nearly follows the derangement numbers, as shown in Proposition \ref{propo24}.

\section*{References}

[1] R.A. Brualdi, Introductory Combinatorics (1992), 2nd edition.

[2] R.P. Stanley, Enumerative Combinatorics, Vol. 1 (2011), 2nd edition.  

\newpage
\begin{center}
{\Large \textbf{APPENDIX}}
\end{center}
\appendix
\begin{table}[h!]
  \centering
  {
  \begin{tabular}{lc}
  \begin{tabular}{|c|c|}
  \hline
    $\ \ \ \{D_2\}=\emptyset$\ \ \ & $\{d_2\}=\{21\}$ \\
    \hline
  \end{tabular} & \\
  \begin{tabular}{|c|}
    \hline
    $\{D_3\}=\{d_3\}$ \\
    \hline
    132 \\
    213 \\
    321 \\
     \\
     \\
     \\
     \\
     \\
    \hline
  \end{tabular} & \hspace{-0.4cm}\begin{tabular}{cc|c|c|}
    \cline{3-4}
     &  & \multicolumn{2}{|c|}{$\{d_4\}=\{d_{41}\} \cup \{D_4\}$} \\
     \hline
    \hspace{-0.21cm}\vline$\hspace{0.5cm}\{D_3\}$ & \vline\hspace{0.4cm}\ $\{D_4\}$\hspace{0.4cm} & \hspace{0.4cm}$\{d_{41}\}$\hspace{0.4cm} & \hspace{0.4cm}$\{D_4\}$\hspace{0.4cm} \\
    \hline
   \hspace{-0.47cm}\vline\hspace{0.4cm} 132 & \vline\hspace{0.42cm} 1324\ \ \ \ \ & 4132 & 1324 \\
   \hspace{-0.47cm}\vline\hspace{0.4cm} 213 & \vline\hspace{0.42cm} 1432\ \ \ \ \  & 2413 & 1432 \\
   \hspace{-0.47cm}\vline\hspace{0.4cm} 321 &\vline\hspace{0.42cm} 2143\ \ \ \ \  & 3241 & 2143 \\
   \hspace{-1.51cm}\vline\hspace{0cm}  & \vline\hspace{0.54cm}2431\ \ \ \ \ &  & 2431 \\
   \hspace{-1.51cm}\vline\hspace{0cm}  & \vline\hspace{0.54cm}3142\ \ \ \ \ &  & 3142 \\
   \hspace{-1.51cm}\vline\hspace{0cm}  & \vline\hspace{0.54cm}3214\ \ \ \ \  &  & 3214 \\
   \hspace{-1.51cm}\vline\hspace{0cm}  & \vline\hspace{0.54cm}4213\ \ \ \ \  &  & 4213 \\
   \hspace{-1.51cm}\vline\hspace{0cm}  & \vline\hspace{0.54cm}4321\ \ \ \ \  &  & 4321 \\
    \hline
  \end{tabular}\\

  \end{tabular}}
  \\[1em]
  {
  \hspace*{0.2cm}\begin{tabular}{|c|c|c|c|c|c|}
    \hline
    $\{D_4\}$ & \multicolumn{5}{|c|}{$\{D_5\}$} \\
    \hline
    \hspace{0.67cm}1324\hspace{0.67cm} & \hspace{0.45cm}13254\hspace{0.45cm} & \hspace{0.3cm}21354\hspace{0.3cm} &\hspace{0.25cm}31425\hspace{0.25cm} & \hspace{0.4cm}41325\hspace{0.4cm} &\hspace{0.35cm}52143\hspace{0.35cm} \\
    1432 & 13524 & 21435 & 31524 & 41352 & 52413 \\
    2143 & 13542 & 21543 & 31542 & 41532 & 52431 \\
    2431 & 14253 & 24135 & 32154 & 42135 & 53142 \\
    3142 & 14325 & 24153 & 32415 & 42153 & 53214 \\
    3214 & 14352 & 24315 & 32541 & 42531 & 53241 \\
    4213 & 15243 & 25314 & 35214 & 43152 & 54132 \\
    4321 & 15324 & 25413 & 35241 & 43215 & 54213 \\
	    & 15432 & 25431 & 35421 & 43521 & 54321 \\
    \hline
  \end{tabular}
\\[1em]
  \hspace*{0.2cm}\begin{tabular}{|c|c|c|c|c|c|}
    \hline
    \multicolumn{6}{|c|}{$\{d_5\}=\{d_{51}\}\cup \{D_5\}$} \\
    \hline
    $\{d_{51}\}$ & \multicolumn{5}{|c|}{$\{D_5\}$} \\
    \hline
    \hspace{0.6cm}51324\hspace{0.6cm} & \hspace{0.4cm}13254\hspace{0.4cm} &\hspace{0.3cm}21354\hspace{0.3cm} &\hspace{0.25cm}31425\hspace{0.25cm} & \hspace{0.45cm}41325\hspace{0.45cm} &\hspace{0.35cm}52143\hspace{0.35cm} \\
    51432 & 13524 & 21435 & 31524 & 41352 & 52413 \\
    25143 & 13542 & 21543 & 31542 & 41532 & 52431 \\
    24351 & 14253 & 24135 & 32154 & 42135 & 53142 \\
    35142 & 14325 & 24153 & 32415 & 42153 & 53214 \\
    32514 & 14352 & 24315 & 32541 & 42531 & 53241 \\
    42513 & 15243 & 25314 & 35214 & 43152 & 54132 \\
    43251 & 15324 & 25413 & 35241 & 43215 & 54213 \\
         & 15432 & 25431 & 35421 & 43521 & 54321 \\
    \hline
  \end{tabular}}
  \caption{Arrangements in $\{D_n\}$ and $\{d_n\}$ up to $n = 5$.}\label{tabA1}
\end{table}

\begin{table}
  \centering
  {
  \begin{tabular}{|r|r|r|r|c|}
    \hline
    $n$\quad & $Der_n\qquad $ & $D_n$\qquad\ & $d_n$\qquad\ & $Der_n - D_n$ \\
    \hline
   0 & 1 & 0 & & $+1$ \\
    1 & 0 & 1 & 1 &  $-1$ \\
    2 & 1 & 0 & 1 & $+1$ \\
    3 & 2 & 3 & 3 &  $-1$ \\
    4 & 9 & 8 & 11 & $+1$ \\
    5 & 44 & 45 & 53 &  $-1$ \\
    6 & 265 & 264 & 309 & 	$+1$ \\
    7 & 1.854 & 1.855 & 2.119 &  $-1$ \\
    8 & 14.833 & 14.832 & 16.687 & $+1$ \\
    9 & 133.496 & 133.497 & 148.329 &  $-1$ \\
    10 & 1.334.961 & 1.334.960 & 1.468.457 & $+1$ \\
    11 & 14.684.570 & 14.684.571 & 16.019.531 &  $-1$ \\
    12 & 176.214.841 & 176.214.840 & 190.899.411 & $+1$ \\
    13 & 2.290.792.932 & 2.290.792.933 & 2.467.007.773 &  $-1$ \\
    14 & 32.071.101.049 & 32.071.101.048 & 34.361.893.981 & $+1$ \\
    15 & 481.066.515.734 & 481.066.515.735 & 513.137.616.783 &  $-1$ \\
    \hline
  \end{tabular}}
  \caption{Values for  $Der_n$, $D_n$ and $d_n$ up to $n = 15$.}\label{tabA2}
\end{table}
\vspace{7cm}\


\end{document}